\documentclass[12pt]{article}
\usepackage{amsmath}
\usepackage{amssymb}
\usepackage{amscd}
\usepackage{amsthm}

\theoremstyle{plain}
\newtheorem{Thm}{Theorem}

\newtheorem{Prop}[Thm]{Proposition}
\newtheorem{Cor}[Thm]{Corollary}
\newtheorem{Lem}[Thm]{Lemma}

\theoremstyle{definition}

\newtheorem{Expl}[Thm]{Example}

\theoremstyle{Remark}

\numberwithin{equation}{section}

\title{Variation of mixed Hodge structures and 
the positivity for algebraic fiber spaces}

\author{Yujiro Kawamata}

\begin{document}

\maketitle

\begin{abstract}
These are the lecture notes based on earlier papers 
with some additional new results. 
New and simple proofs are given for local freeness theorem
and the semipositivity theorem.
A decomposition theorem for higher direct images of dualizing sheaves 
of Koll\'ar is extended to the sheaves of differential forms 
of arbitrary degrees in the case of a well 
prepared birational model.
We will also prove the log versions of some of the results, i.e., the case 
where we allow horizontal boundary components.
\end{abstract}

These are the lecture notes of a mini-course given 
at Institut Henri Poincar\'e 
for the thematic term \lq\lq Complex Algebraic Geometry'' in May 2010.
The author would like to thank the organizers Olivier Debarre, 
Christoph Sorger and Claire Voisin.

This is based on my papers \cite{AF} and \cite{Reducible} with some additional 
new results.
New and simple proofs are given for local freeness theorem 
(proved in \cite{AB} in the case of the direct image sheaf and 
in \cite{Kollar} and \cite{Nakayama} in the general case), 
and thereby the semipositivity theorem \cite{AB}.
A decomposition theorem for higher direct images of dualizing sheaves 
of Koll\'ar \cite{Kollar} is extended to the sheaves of differential forms 
of arbitrary degrees in the case of a well 
prepared birational model.
We note that we have to consider differential forms of all degrees 
in order to use to the full power of the Hodge theory.

An algebraic fiber space is the relative version of an algebraic variety.
A good birational model of an algebraic variety is a smooth model.
The corresponding good model of an algebraic fiber space is a 
{\em weak semistable model} found by Abramovich and Karu \cite{AK}.
A {\em well prepared model} is a birational model which satisfies the
same conditions as a weak semistable model except that all the fibers 
are reduced.
One can obtain a weak semistable model from a well prepared model by a 
finite and flat base change.

From a well prepared model, we shall construct a good topological model 
by applying a method of log geometry (\cite{Kato}).
We associate a topological space $X^{\log}$, called a {\em log space}, 
to a log pair $(X,B)$ 
consisting of a complex analytic variety with a boundary divisor.
The log space $X^{\log}$ can be regarded as a compactification 
of the open space $X \setminus B$.
It is a real analytic manifold with boundary and corners.
This compactification is homeomorphic to a closed subset of 
$X \setminus B$, the complement of a tubular neighborhood 
of $B$.
But because we have points of the log space 
$X^{\log}$ above the boundary divisor $B$, 
a complicated limiting argument becomes unnecessary 
in the proofs of our results.
For example, the monodromy actions around the singular fibers 
become very easily handled on the log space, 
because the singularities of degenerate fibers disappear in this 
topological model, and the fibers over the boundary have stratifications.

We define the structure sheaf of a log space in a way that it contains the 
logarithms of local coordinates, and then the sheaves of differential forms
that they contain logarithmic differentials.
We prove our main result on the existence of 
a cohomological mixed Hodge complex on any fiber, 
a concept defined by Deligne \cite{Deligne-II}.
As corollaries, we obtain $E_1$ and $E_2$ degenerations of spectral sequences
with respect to the Hodge and weight filtrations respectively.

The construction of this paper is as follows.
We recall some preliminary results in \S\S 1 to 5, 
the weak semistable reduction theorem of Abramovich and Karu in \S1,
the construction of a topological space associated to a log scheme in \S2, 
a criteria for the decomposition of an object in a derived category 
due to Deligne in \S3, 
the definitions related to the cohomological mixed Hodge complex 
due also to Deligne in \S4, 
and the Griffiths positivity theorem and the semi-simplicity theorem 
of the category of variations of polarized Hodge structures in \S5. 

We shall prove our results in \S\S 6 to 8 on the existence of a mixed 
Hodge complex.
We shall construct weight filtrations on the sheaves on the log spaces in \S6.
In order to compare the weight filtrations on the topological leval and
the De Rham level, we introduce the sheaves of differential forms in 
\S7 and prove Poincar\'e lemmas.
Our construction is justified in \S8 that the weight filtration and the 
Hodge filtration defined on a singular fiber define a 
cohomological mixed Hodge complex.
The degenerations of spectral sequences follow from the general machinery 
explained in \S4.
The local freeness follows as a corollary.

In \S9 we prove the semipositivity theorem.
First we treat the case where the fiber space is weakly semistable and there 
is no horizontal boundary components follwoing \cite{AB}.
This is a consequence of the 
Griffiths positivity theorem for the open part and the boundary considerations.
Then the theorem is extended to the case of a well prepared fiber space
which may have horizontal boundary components.
We prove decomposition theorems in the final \S10 for the higher direct 
images of the constant sheaf and the sheaf of logarithmic differential forms. 

We work over $\mathbf{C}$ in this paper.
The topology of an algebraic variety we consider is the analytic topology
of the underlying complex analytic space.

The author would like to thank Taro Fujisawa for pointing out 
a serious error in the previous version of this paper.

\section{Weak semistable reduction}

We recall a weak semistable reduction theorem of Abramovich and Karu \cite{AK}.
A general reference for toroidal varieties is \cite{KKMS}.

A {\em toroidal variety} $(X,B)$ is a pair consisting of a normal variety
and an effective reduced divisor such that each point $x \in X$ has a 
{\em toric local model} in the following sense:
there is an complex analytic neighborhood $x \in U$ such that the pair 
$(U,B \vert_U)$ is complex analytically isomorphic to another pair
$(U',B'\vert_{U'})$ which comes from a toric variety $(X',B')$,
a normal variety with an action of an algebraic torus $X' \setminus B'$.
We assume moreover that the pair is {\em strict} or 
{\em without self-intersection} in the sense that each irreducible component
of $B$ is normal.

A toroidal variety $(X,B)$ is said to be {\em smooth} if 
$X$ is smooth and $B$ has only normal crossings.
It is {\em quasi-smooth} if there exists a local toric model of each point 
which has only abelian quotient singularities.

A {\em toroidal morphism} $f: (X,B) \to (Y,C)$ between toroidal varieties
is one which has a toric local model at each point $x \in X$ 
in the following sense: there is a toric morphism between local models 
$f': (X',B') \to (Y',C)$, i.e., $f' \vert_{X' \setminus B'}:
X' \setminus B' \to Y' \setminus C'$ is a surjective homomorphism of 
algebraic tori, and $f'$ is equivariant under the torus actions.

A toric model at each point $x \in X$ is described by a rational 
polyhedral closed convex cone 
$\sigma_x$ in a finite dimensional real vector space with an integral lattice.
By gluing these cones, we construct a {\em fan} $\Delta_{X,B}$.
Unlike the toric case, the fan is not embedded in a fixed real vector space,
and the pair $(X,B)$ cannot be reconstructed from the fan.
But the local structure of the pair as well as its birational
modifications are completely described by the fan. 

The pair $(X,B)$ is quasi-smooth if and only if 
each cone $\sigma_x$ is simplicial,
and smooth if and only if 
in addition that the lattice points on the edges of the 
cone generate a saturated subgroup of the lattice.
A rational subdivision of the fan corresponds to a birational modification.

\begin{Expl}
Let $f: (X,B) \to (Y,C)$ be a toroidal morphism between smooth toroidal
varieties.
Take an arbitrary point $x \in X$ and its image $y = f(x)$.
Then there exist local coordinates $(x_1,\dots,x_n)$ and $(y_1,\dots,y_m)$
around $x$ and $y$ respectively such that we can write
$f^*y_i = \prod_j x_j^{r_{ij}}$, where the $r_{ij}$ are non-negative 
integers such that $r_{ij} \ne 0$ for at most one $i$ for each $j$.

More generally, if $(X,B)$ is a quasi-smooth variety instead of a smooth one,
then there is an analytic neighborhood $U \subset X$ of the given point 
$x \in X$ and a finite Galois toroidal covering 
$\pi: (X',B') \to (U,B \vert_U)$ from a smooth toroidal variety with 
a point $x' \in X'$ such that $\pi(x') = x$.
There exist local coordinates $(x'_1,\dots,x'_n)$ around $x' \in X$ which are 
semi-invariant with respect to the Galois action.
Therefore we have a similar local expression also in this case.
\end{Expl}

The following theorem of Abramovich and Karu gives a well prepared birational 
model of an algebraic fiber space:

\begin{Thm}[\cite{AK}]
Let $f_0: X_0 \to Y_0$ be a surjective morphism of projective varieties 
with geometrically connected fibers defined over a field of characteristic 
zero, and $Z$ a closed subset of $X_0$.
Then there exist a quasi-smooth projective toroidal variety $(X,B)$, 
a smooth projective toroidal variety $(Y,C)$, a projective morphism 
$f: X \to Y$ with geometrically connected fibers,
and projective birational morphisms $\mu_X: X \to X_0$, $\mu_Y: Y \to Y_0$ 
such that $\mu_Y \circ f = f_0 \circ \mu_X$ 
and which satisfy the following conditions:

(1) $f: (X,B) \to (Y,C)$ is a toroidal morphism.

(2) All the fibers of $f$ have the same dimension.

(3) $\mu_X^{-1}(Z) \subset B$.

Moreover there exists another smooth projective toroidal variety 
$(Y',C')$ and a finite surjective morphism $\pi: Y' \to Y$ 
such that $C' = \pi^{-1}(C)$, and satisfy the
following condition:

(4) The induced morphism $f': (X',B') \to (Y',C')$ 
for the normalization $(X',B')$ of the fiber
product $(X,B) \times_Y Y'$ is still a toroidal morphism 
from a quasi-smooth projective toroidal variety 
to a smooth projective toroidal variety, such that
all the fibers of $f'$ are reduced.
\end{Thm}

We note that the covering morphism $\pi$ is flat because 
$Y$ is smooth.
The discriminant locus of $\pi$ is a normal crossing divisor which
properly contains $C$.
Therefore $\pi$ is not necessarily a toroidal morphism.

\begin{proof}[Idea of proof]
The proof is based on the idea of De Jong's alteration.
Because of the assumption on the characteristic of the base field, 
the ramification theory is simple and the singularities can be resolved.

First we prove the existence of a birational model of 
$f_0: X_0 \to Y_0$ which is a toroidal morhism between 
smooth toroidal varieties.
We proceed by induction on the relative dimension $\dim X_0 - \dim Y_0$.
By using the generic projection, we decompose the morphism $f_0$ to 
two morphisms $g: X_0 \to X_1$ and $h: X_1 \to Y_0$ such that 
$\dim X_0 = \dim X_1 + 1$, and set $Z_1 = \emptyset$.
We replace a birational model of $h: X_1 \to Y_0$ by the induction assumption,
while we prepare the morphism $g: X_0 \to X_1$ by the semistable reduction 
using the moduli space of pointed curves.

Secondly we modify the toroidal model by using the fans associated to the
toroidal structure.
The morphism becomes equi-dimensional when we subdivide the fans $\Delta_X$
and $\Delta_Y$ of $X$ and $Y$
so that the images of $1$-dimensional cones in $\Delta_X$ become 
$1$-dimensional cones in $\Delta_Y$.
In order to make the fibers reduced, we apply the covering trick in
\cite{AB}.
\end{proof}

We call the morphism $f: (X,B) \to (Y,C)$ a 
{\em well prepared birational model} of an algebraic fiber space
$f_0:X_0 \to Y_0$,
and its finite base change $f': (X',B') \to (Y',C')$ 
a {\em weak semistable model}.
The latter is \lq\lq weak'' because the pair $(X',B')$ 
has some mild singularities, i.e., abelian quotient singularities.
But these singularities cause no trouble as explained 
in the later sections.

\section{Log space}

The idea of the real oriented blowing-ups
of the (quasi-)smooth toroidal pairs
and the application to semistable degenerations of varieties 
can go back at least to a book \cite{Persson}.
A general reference of this section is \cite{Kato} and \cite{KN}.

A {\em log scheme} $(X,M,\alpha)$ 
consists of a scheme $X$, a sheaf 
of semi-groups with unit $M$, and a semi-group homomorphism
$\alpha: M \to \mathcal{O}_X$ such that 
$\alpha^{-1}(\mathcal{O}_X^*) \cong \mathcal{O}_X^*$, where 
the semi-group structure of $\mathcal{O}_X$ is given by the 
multiplication.

A {\em morphism} of log schemes $(f,\phi): (X,M,\alpha) \to 
(Y,N,\beta)$ consists of a morphism of schemes
$f: X \to Y$ and a semi-group homomorphism
$\phi: f^{-1}N \to M$ such that $\alpha \circ \phi =
f^* \circ f^{-1}\beta$.

\begin{Expl}
(1) Let $T = \text{Spec }\mathbf{C}$, and 
$M_T = \mathbf{R}_{\ge 0} \times S^1$.
We define $\alpha_T: M_T \to \mathbf{C}$ by
$\alpha_T(r,\theta) = re^{i\theta}$.
Then $(T,M_T,\alpha_T)$ is a log scheme.

(2) If $(X,B)$ is a toroidal variety, 
then the subsheaf of $\mathcal{O}_X$ given by
\[
M = \{h \in \mathcal{O}_X \,\vert\, h \vert_{X \setminus B} \in 
\mathcal{O}_X^*\}
\]
defines a log structure on $X$.
\end{Expl}

Let $(X,M,\alpha)$ be a log scheme defined over $\mathbf{C}$.
We define the {\em associated log space} 
$X^{\log}$ to be the set of all log morphisms 
$(f,\phi): (T,M_T,\alpha_T) \to (X,M,\alpha)$ from the log scheme defined in 
the above example.
Let $\rho: X^{\log} \to X$ be the natural map defined by
$\rho(f,\phi) = \text{Im}(f)$.

\begin{Expl}
If $(X,B) = (\mathbf{C},0)$, then 
$X^{\log} = \mathbf{R}_{\ge 0} \times S^1$ and
$\rho(r,\theta) = re^{i\theta}$.

Indeed let $t$ be a coordinate on $X$ and assume that $f(T) = t_0$.
If $t_0 \ne 0$, then the image of $t \in M_X$ is uniquely determined.
If $t_0 = 0$, then the image of $t$ is of the form $(0,\theta)$. 
\end{Expl}

If $(X,B)$ is a quasi-smooth toroidal variety, 
then the set $X^{\log}$ has a structure of a manifold with boundary and corners as shown in the following proposition.
It is a kind of a compactification of the open variety $X \setminus B$:

\begin{Prop}
Let $(X,B)$ be a quasi-smooth toroidal variety.
Then the following hold:

(1) The associated log space $X^{\log}$ has a structure of 
a real analytic manifold with boundary and corners, 
and $\rho: X^{\log} \to X$ is a proper continuous map of topological spaces.

(2) The complex manifold $X \setminus B$ is homeomorphic to
a dense open subset of $X^{\log}$.
There is a small analytic open neighborhood $U$ of $B \subset X$
such that $X \setminus U$ is homeomorphic to $X^{\log}$.
\end{Prop} 

\begin{proof}
(1) If $(X,B) = (X_1 \times X_2, p_1^*B_1 + p_2^*B_2)$ for 
smooth toroidal varieties $(X_i,B_i)$ with associated
log spaces $X_i^{\log}$ for $i=1,2$, then
we have $X^{\log} \cong X_1^{\log} \times X_2^{\log}$ and
$\rho \cong \rho_1 \times \rho_2$.

If $(X,B)$ is a quotient of a smooth toroidai variety $(X',B')$ by a 
finite abelian group $G$ acting on $X'$ leaving $B'$ preserved, then 
the associated log space $X^{\log}$ is the quotient of 
$(X')^{\log}$ by the induced group action of $G$ which is fixed point free.

Since $(X,B)$ is covered by standard open subsets as above and the set 
$X^{\log}$ is globally defined, the log spaces associated to this covering 
are glued to yield a manifold with boundary and corners.

(2) This is because the map $\rho$ induces a bijection 
$\rho^{-1}(X \setminus B) \to X \setminus B$.
\end{proof}

\begin{Cor}
Let $f: (X,B) \to (Y,C)$ be a well prepared toroidal morphism.
Then the associated morphism of log spaces 
$f^{\log}: X^{\log} \to Y^{\log}$ 
is a locally trivial topological fiber bundle.
In particular, the higher direct image sheaves 
$R^pf^{\log}_*\mathbf{Z}_{X^{\log}}$
are locally constant sheaves on $Y^{\log}$.
\end{Cor}

\begin{proof}
It follows from the fact that the induced morphism 
$f: X \setminus B \to Y \setminus C$ is a 
locally trivial topological fiber bundle.
\end{proof}

The important advantage of $X^{\log}$ over $X \setminus B$
is that the geometric objects are extended over the singular fibers 
so that we do not need limiting argument.

\section{Criteria for decompositions}

We recall Deligne's criteria for an object of a derived category 
to be decomposed into its cohomology groups. 

\begin{Prop}[\cite{Deligne-L}]
Let $x \in D^b(A)$ be an object of a bounded derived category 
of an abelian category $A$.
Then the following conditions are equivalent:

(1) $x \cong \bigoplus_i H^i(x)[-i]$ in $D^b(A)$.

(2) For an arbitrary exact functor $T : D^b(A) \to D(B)$ 
to the derived category of an arbitrary
abelian category $B$, the Grothendieck-Verdier spectral sequence
\[
E_2^{p,q} = H^p(T(H^q(x))) \Rightarrow H^{p+q}(T(x))
\]
degenerate at $E_2$.
\end{Prop}

\begin{proof}
The condition (1) implies that 
$H^k(T(x)) \cong \bigoplus_{i+j=k} H^i(T(H^j(x)))$, hence (2).

We prove the converse (2) $\Rightarrow$ (1).
Let $T(y) = \bigoplus_j \text{Hom}(H^i(x),y[j])[-j] \in 
D(\mathbf{Z}\text{-mod})$ for a fixed $i$.
By assumption, the spectral sequence
\[
E_2^{p,q} = \text{Hom}(H^i(x), H^q(x)[p]) 
\Rightarrow \text{Hom}(H^i(x), x[p+q])
\]
degenerate at $E_2$.
In particular, the edge homomorphism 
$E_{\infty}^i \to E_2^{0,i}$ given by
\[
\text{Hom}(H^i(x), x[i]) \to \text{Hom}(H^i(x), H^i(x))
\]
is surjective.
Let $f_i: H^i(x)[-i] \to x$ be an arbitrary morphism which is 
mapped to the identity morphism of $H^i(x)[-i]$.
Then the sum $\bigoplus_i f_i: \bigoplus_i H^i(x)[-i] \to x$ is
a quasi-isomorphism.
\end{proof}

\begin{Thm}[\cite{Deligne-L}]\label{LD}
Let $x \in D^b(A)$ be an object of a bounded derived category 
of an abelian category $A$, $u \in \text{Hom}(x,x[2])$, and 
$n$ an integer.
Assume that the induced morphisms $u^i: H^{n-i}(x) \to H^{n+i}(x)$
are isomorphisms for all $i \ge 0$.
Then $x \cong \bigoplus_i H^i(x)[-i]$ in $D^b(A)$.
\end{Thm}

\begin{proof}
We define primitive parts of $H^{n-i}(x)$ for $i \ge 0$ 
by 
\[
{}_0H^{n-i}(x) = 
\text{Ker}(u^{i+1}: H^{n-i}(x) \to H^{n+i+2}(x)).
\]
Then we have a Lefschetz decomposition
\[
\begin{split}
&H^{n-i}(x) \cong \bigoplus_{k \ge 0} u^k{}_0H^{n-i-2k}(x) \\
&H^{n+i}(x) \cong \bigoplus_{k \ge 0} u^{k+i}{}_0H^{n-i-2k}(x).
\end{split}
\]
Let $T: D^b(A) \to D(B)$ be an arbitrary exact functor, and
\[
E_2^{p,q} = H^p(T(H^q(x))) \Rightarrow H^{p+q}(T(x))
\]
be the spectral sequence.
We shall prove that the differentials 
$u_r^{p,q}: E_r^{p,q} \to E_r^{p+r,q-r+1}$ vanish for all 
$r \ge 2$ by induction on $r$.
Assume that we already have $E_2^{p,q} \cong E_r^{p,q}$.
Denote ${}_0E_r^{p,q} = {}_0E_2^{p,q} = H^p(T({}_0H^q(x)))$.
We have the following commutative diagram
\[
\begin{CD}
{}_0E_r^{p,n-i} @>{d_r^{p,n-i}}>> E_r^{p+r,n-i-r+1} \\
@V{u^{i+1}}VV @VV{u^{i+1}}V \\
E_r^{p,n+i+2} @>{d_r^{p,n+i+2}}>> E_r^{p+r,n+i-r+3}
\end{CD}
\]
The left vertical arrow vanishes while the right vertical arrow 
is a monomorphism since $r \ge 2$.
Therefore $d_r^{p,n-i}$ vanishes on ${}_0E_r^{p,n-i}$.
Since $u^kd_r^{p,n-i} = d_r^{p,n-i+2k}u^k$, we have $d_r^{p,q} = 0$
for all p,q.
\end{proof} 

\section{Hodge structure and Hodge complex}

We recall the definitions concerning Hodge structures and Hodge complexes.
The reference is \cite{Deligne-II}.

A {\em Hodge structure (HS)} of weight $n$ is a pair
$(H_{\mathbf{Z}}, F)$, where $H_{\mathbf{Z}}$ 
is a finitely generated $\mathbf{Z}$-module and $F$ is a
decreasing filtration on $H_{\mathbf{C}} = H_{\mathbf{Z}}
\otimes \mathbf{C}$ which induces a direct sum decomposition
\[
H_{\mathbf{C}} = \bigoplus_{p+q=n}
F^p(H_{\mathbf{C}}) \cap \overline{F^q(H_{\mathbf{C}})}.
\]
$F$ is called a {\em Hodge filtration}.
We set $H^{p,q} = F^p(H_{\mathbf{C}}) \cap \overline{F^q(H_{\mathbf{C}})}$.

It is said to be 
{\em polarizable} if there is a non-degenerate bilinear form
$Q: H_{\mathbf{R}} \times H_{\mathbf{R}} \to \mathbf{R}$
for $\mathbf{R} = H_{\mathbf{Z}} \otimes \mathbf{R}$
such that the following Hodge-Riemann index theorem holds:
\[
\begin{split}
&i^w(-1)^{q+1}Q(H^{p,q}, \bar H^{p,q}) >> 0 \\
&Q(H^{p,q}, \bar H^{p',q'}) = 0 \text{ if } p \ne p'
\end{split}
\]
for all $p+q=p'+q'=n$.

A {\em mixed Hodge structure (MHS)} is a triple consisting of a 
finitely generated $\mathbf{Z}$-module $H_{\mathbf{Z}}$, an increasing
filtration $W$ on $H_{\mathbf{Q}} = H_{\mathbf{Z}}
\otimes \mathbf{Q}$, and a
decreasing filtration $F$ on $H_{\mathbf{C}} = H_{\mathbf{Z}}
\otimes \mathbf{C}$ such that $(\text{Gr}_n^W(H_{\mathbf{Q}}),F)$ 
becomes a Hodge $\mathbf{Q}$-structure of weight $n$, a Hodge 
structure tensorized by $\mathbf{Q}$,
where $F$ denotes the induced filtration from $F$ by abuse of notation.

$W$ is called a {\em weight filtration}.
It is {\em graded polarzable} if each graded piece 
$\text{Gr}_n^W(H_{\mathbf{Q}})$ is individually polarizable.

\begin{Expl}
Let $X$ be a smooth projective variety of dimension $n$.
Then the cohomology group $H^k(X, \mathbf{Z})$ carries a Hodge structure
of weight $k$.
If we take its primitive part, then we obtain a polarizable Hodge structure 
as follows.

Let $L \in H^2(X, \mathbf{Z})$ 
be the cohomology class of an ample divisor.
For each integer $i \ge 0$, the cup product with $L^i$ yields
an isomorphism $L^i: H^{n-i}(X, \mathbf{Q}) \to H^{n+i}(X, \mathbf{Q})$.
The {\em primitive part} ${}_0H^{n-i}(X, \mathbf{Q}) 
\subset H^{n-i}(X, \mathbf{Q})$ is defined as the kernel 
of the homomorhism given by the cup product with $L^{i+1}$.
The {\em Lefschetz decomposition theorem} says that 
there is a direct sum decomposition
\[
H^k(X, \mathbf{Q}) = \bigoplus_{i \ge 0} L^i{}_0H^{k-2i}(X, \mathbf{Q})
\]
for $k \le n$.
The Hodge-Riemann index theorem says that the primitive part 
${}_0H^k(X, \mathbf{Q})$ carries a polarizable 
$\mathbf{Q}$-Hodge structure of weight $k$.
The polarization comes from the cup product: 
\[
Q(u,v) = (u \cup v \cup L^{n-k})[X]
\]
for $u,v \in {}_0H^k(X, \mathbf{Q})$.
\end{Expl}

A {\em Hodge complex (HC)} of weight $n$ is a triple 
$H = (H_{\mathbf{Z}}, H_{\mathbf{C}}, F)$, where 
$H_{\mathbf{Z}} \in D^+(\mathbf{Z})$ is a left bounded complex of 
abelian groups whose cohomology groups $H^k(H_{\mathbf{Z}})$ 
are of finite type, and 
$H_{\mathbf{C}} \in D^+F(\mathbf{C})$ is a left bounded complex of 
$\mathbf{C}$-modules
with a decreasing filtration $F$ such that the following conditions 
are satisfied:

(1) There is an isomorphism
$H_{\mathbf{Z}} \otimes \mathbf{C} \cong H_{\mathbf{C}}$.

(2) The spectral sequence associated to the filtration $F$
\[
{}_FE_1^{p,q} = H^{p+q}(\text{Gr}_F^p(H_{\mathbf{C}})) \Rightarrow
H^{p+q}(H_{\mathbf{C}})
\]
degenerates at $E_1$.

(3) The cohomology groups $H^k(H_{\mathbf{Z}})$ 
with the filtrations on the $H^k(H_{\mathbf{C}})$ induced from $F$ 
give Hodge structures of weight $n+k$
for all $k$.

If $H$ is a Hodge complex of weight $n$, then
the shifted complex $H[m]$ is a Hodge complex of weight $n+m$.

A {\em mixed Hodge complex (MHC)} is a quadruple 
$H = (H_{\mathbf{Z}}, H_{\mathbf{C}}, W, F)$, where 
$H_{\mathbf{Z}} \in D^+(\mathbf{Z})$ is a left bounded complex of 
abelian groups with an increasing filtration $W$ on 
$H_{\mathbf{Q}} = H_{\mathbf{Z}} \otimes \mathbf{Q}$, 
and $H_{\mathbf{C}} \in D^+F_2(\mathbf{C})$ is a left bounded complex of 
$\mathbf{C}$-modules
with an increasing filtration $W$ and a decreasing filtration $F$ 
such that the following conditions 
are satisfied:

(1) There is an isomorphism
$(H_{\mathbf{Q}},W) \otimes \mathbf{C} \cong (H_{\mathbf{C}},W)$.

(2) $(\text{Gr}_n^W(H_{\mathbf{Q}}),F)$ is a Hodge $\mathbf{Q}$-complex 
of weight $n$, a Hodge complex tensorized by $\mathbf{Q}$,
where $F$ denotes the induced filtration from $F$ by abuse of notation.

Then it follows that the triple $(H^k(H_{\mathbf{Z}}),W[k],F)$ is a 
mixed Hodge structure.

A {\em cohomological Hodge complex (CoHC)} of weight $n$ on a 
topological space $X$ is a triple 
$H = (H_{\mathbf{Z}}, H_{\mathbf{C}}, F)$, where 
$H_{\mathbf{Z}} \in D^+(\mathbf{Z}_X)$ is a left bounded complex of 
constructible sheaves of abelian groups on $X$, and 
$H_{\mathbf{C}} \in D^+F(X)$ is a left bounded complex of sheaves of
$\mathbf{C}$-modules on $X$ with a decreasing filtration $F$ 
such that the following conditions 
are satisfied:

(1) There is an isomorphism
$H_{\mathbf{Z}} \otimes \mathbf{C} \cong H_{\mathbf{C}}$.

(2) The direct image complex $R\Gamma(X,H) = 
(R\Gamma(X,H_{\mathbf{Z}}), R\Gamma(X,H_{\mathbf{C}},F))$ is a 
Hodge complex of weight $n$.

A {\em cohomological mixed Hodge complex (CoMHC)} on a 
topological space $X$ is a quadruple 
$H = (H_{\mathbf{Z}}, H_{\mathbf{C}}, W, F)$, where 
$H_{\mathbf{Z}} \in D^+(\mathbf{Z}_X)$ is a left bounded complex of 
constructible sheaves of abelian groups on $X$ 
with an increasing filtration $W$ on 
$H_{\mathbf{Q}} = H_{\mathbf{Z}} \otimes \mathbf{Q}$, 
and $H_{\mathbf{C}} \in D^+F_2(X)$ is a left bounded complex of 
sheaves of $\mathbf{C}$-modules on $X$ with an increasing filtration $W$ and 
a decreasing filtration $F$ 
such that the following conditions 
are satisfied:

(1) There is an isomorphism
$(H_{\mathbf{Q}},W) \otimes \mathbf{C} \cong (H_{\mathbf{C}},W)$.

(2) $(\text{Gr}_n^W(H_{\mathbf{Q}}),\text{Gr}_n^W(H_{\mathbf{C}}),F)$ 
is a cohomological 
Hodge $\mathbf{Q}$-complex weight $n$ on $X$.

In this case, it follows that 
the direct image complex 
\[
R\Gamma(X,H) = 
(R\Gamma(X,H_{\mathbf{Z}}), R\Gamma(X,H_{\mathbf{C}}),W,F)
\]
becomes a mixed Hodge complex.

We summarize the above definitions as follows:
\[
\begin{CD}
\text{CoMHC} @>{R\Gamma}>> \text{MHC} @>{H^*}>> \text{MHS} \\
@V{\text{Gr}^W}VV @V{\text{Gr}^W}VV @V{\text{Gr}^W}VV \\
\text{CoHC} @>{R\Gamma}>> \text{HC} @>{H^*}>> \text{HS}
\end{CD}
\]

\begin{Expl}
Let $X$ be a smooth projective variety, $B$ a normal crossing divisor, and 
$i: X \setminus B \to X$ the open immersion.
Then the direct image sheaf $H = Ri_*\mathbf{Z}_{X \setminus B}$
carries a structure of a cohomological mixed Hodge complex.
We can also write $H = R\rho_*\mathbf{Z}_{X^{\log}}$ for the real oriented 
blow-up $\rho: X^{\log} \to X$.

The weight filtration on $H_{\mathbf{Q}}$ on the $\mathbf{Q}$-level 
is given by the 
canonical filtration
\[
W_q(H_{\mathbf{Q}}) = \tau_{\le q}(Ri_*\mathbf{Z}_{X \setminus B})
\]
where the truncation $\tau_{\le q}(K^{\bullet})$ of a 
complex $K^{\bullet}$ is defined by
\[
\tau_{\le q}(K^{\bullet})_i = 
\begin{cases} K_i &\text{ if } i < q \\
\text{Ker}(K_q \to K_{q+1}) &\text{ if } i=q \\
0 &\text{ if } i > q.
\end{cases}
\]
The $\mathbf{C}$-level complex is given by the sheaves of 
logarithmic differential forms
$H_{\mathbf{C}} = \Omega^{\bullet}_X(\log B)$.
We have a quasi isomorphism
\[
Ri_*\mathbf{C}_{X \setminus B} \cong \Omega^{\bullet}_X(\log B)
\]
by the Poincar\'e lemma.
The weight filtration $W$ on the $\mathbf{C}$-level 
is given by the order of log poles;
$W_q(H_{\mathbf{C}})$ is a complex consisting of differential forms
whose log poles have order at most $q$.
The residue homomorphisms give an isomorphism of filtered complexes
\[
(H_{\mathbf{Q}},W) \otimes \mathbf{C} \cong (H_{\mathbf{C}},W).
\]
The Hodge filtration is given by the stupid filtration:
\[
F^p(H_{\mathbf{C}}) = \Omega^{\ge p}_X(\log B).
\]

For each non-negative integer $t$, let $B^{[t]}$ be the disjoint union of the 
irreducible components of intersections of $t$ different irreducible 
components of $B$.
We regard $B^{[0]} = X$, and $B^{[1]}$ is the normalization of $B$.
$B^{[t]}$ is $\dim X - t$ equi-dimensional.
The following isomorphisms given by the residue homomorphisms are fundamental:
\[
\begin{split}
&\text{Gr}_q^W(\Omega^{\bullet}_X(\log B)) \cong 
\Omega^{\bullet}_{B^{[q]}}[-q] \\
&F^p(\text{Gr}_q^W(\Omega^{\bullet}_X(\log B))) \cong 
\Omega^{\ge p-q}_{B^{[q]}}[-q].
\end{split}
\]
We denote by $F(-q)$ the shifted filtration defined by
$F(-q)^p=F^{p-q}$.
Then the triple
\[
(\mathbf{Z}_{B^{[q]}}[-q], \Omega^{\bullet}_{B^{[q]}}[-q], F(-q)[-q])
\]
is a cohomological Hodge complex of weight $-q+2q=q$, where the shift of $F$
is counted twice because the opposite filtration $\bar F$ is also shifted.
Therefore the quadruple
\[
(Ri_*\mathbf{Z}_{X \setminus B}, \Omega^{\bullet}_X(\log B), W, F)
\]
is a cohomological mixed Hodge complex.
\end{Expl}

\begin{Thm}[\cite{Deligne-II}]
Let $H$ be a cohomological mixed Hodge complex on a topological space $X$.
Then the following hold:

(1) The spectral sequence associated to the Hodge filtration
\[
{}_FE_1^{p,q} = H^{p+q}(\text{Gr}_F^p(H_{\mathbf{C}})) \Rightarrow
H^{p+q}(H_{\mathbf{C}})
\]
degenerates at $E_1$.

(2) The spectral sequence associated to the weight filtration
\[
{}_WE_1^{p,q} = H^{p+q}(\text{Gr}^W_{-p}(H_{\mathbf{Q}})) \Rightarrow
H^{p+q}(H_{\mathbf{Q}})
\]
degenerates at $E_2$.
\end{Thm}

\section{Semipositivity and semi-simplicity}

Let $H$ be a locally free sheaf on a complex manifold $Y$
with a $C^{\infty}$ hermitian metric $h_H$.
Then there is a connection $\nabla_H: H \to H \otimes \Omega^1_Y$
which is compatible with the holomorphic structure of $H$ and the metric $h_H$.
Let $\Theta_H$ be the curvature form.
It is a $\text{Hom}(H,H)$ valued $C^{\infty}$ differential forms of 
type $(1,1)$ given by the formula
\[
\bar{\partial} \partial h_H(u,v) 
= -h_H(\nabla_H u, \nabla_H v) + h_H(\Theta_H u, v)
\]
where $u,v$ are holomorphic local sections of $H$ and the 
equality holds as $C^{\infty}$ differential form of 
type $(1,1)$.
The first Chern class $c_1(H)$ is given by the $C^{\infty}$ differential form
$\frac i{2\pi} \text{Tr}(\Theta)$.

Let $F$ be a locally free subsheaf of $H$ with an injection $i: F \to H$ and 
a surjection $p: H \to G$ to the quotient sheaf $G = H/F$.
The hermitian metric $h_H$ induces hermitian metrics $h_F$ and 
$h_G$ on $F$ and $G$ respectively.
We can compare the curvature forms of these metrics:

\begin{Prop}\label{secondff}
Define the second fundamental form by 
\[
b = p \circ \nabla_H \circ i \in \text{Hom}(F,G) \otimes \Omega^1_Y
\]
and let $u,v$ be local holomorphic sections of $F$.
Then 
\[
h_F(\Theta_F u,v) = h_H(\Theta_H u,v) - h_G(b(u),b(v)).
\]
\end{Prop}

\begin{proof}
Omitted.
\end{proof}

A {\em variation of mixed Hodge structures (VMHS)} over a
complex manifold $Y$ consists of a locally constant sheaf
$H_{\mathbf{Z}}$, an increasing filtration $W$ by locally constant
subsheaves on a locally constant sheaf $H_{\mathbf{Q}}
= H_{\mathbf{Z}} \otimes \mathbf{Q}$, and a decreasing filtration $F$ 
by locally free subsheaves on a locally free sheaf 
$H_{\mathbf{C}} = H_{\mathbf{Z}} \otimes \mathcal{O}_Y$ 
which satisfy the following conditions:

(1) For each $y \in Y$, the data induced on the fiber 
$(H_{\mathbf{Z},y},(H_{\mathbf{Q},y},W),
(H_{\mathbf{C}} \otimes \mathbf{C}_y,W,F))$ 
is a mixed Hodge structure.  

(2) The connection, called the Gauss-Manin connection
\[
\nabla: H_{\mathbf{C}} \to H_{\mathbf{C}} \otimes \Omega^1_Y
\]
for which sections of $H_{\mathbf{Z}}$ are flat, satisfies the
Griffiths transversality:
\[
\nabla: F^p(H_{\mathbf{C}}) \to F^{p-1}(H_{\mathbf{C}}) 
\otimes \Omega^1_Y 
\]
for all $p$.

\begin{Expl}
Let $f: (X,B) \to (Y,C)$ be a well prepared 
toroidal morphism, and $n$ an integer.
We denote by $f^o: X \setminus f^{-1}(C) \to Y \setminus C$ the restriction.
Let
\[
\begin{split}
&\Omega^1_{X/Y}(\log) = \Omega^1_X(\log B)/f^*\Omega^1_Y(\log C) \\
&\Omega^p_{X/Y}(\log) = \bigwedge^p \Omega^1_{X/Y}(\log).
\end{split}
\]
Then $R^nf^o_*\mathbf{Z}_{X \setminus f^{-1}(C)}$ carries a
structure of VMHS, and
the associated Hodge to de Rham spectral sequence
\[
E_1^{p,q} = R^qf^o_*\Omega_{X/Y}^p(\log) \Rightarrow
R^{p+q}f^o_*(f^o)^{-1}\mathcal{O}_{Y \setminus C}
\]
degenerates at $E_1$ (\cite{Deligne-II}), and gives the 
Hodge filtration.
\end{Expl}

If $W_q(H_{\mathbf{Q}}) = H_{\mathbf{Q}}$ and 
$W_{q-1}(H_{\mathbf{Q}}) = 0$, then our data
is called a {\em variation of Hodge structures (VHS) of
weight $q$}.

A VHS is said to be {\em polarizable} if there is a 
locally constant symmetric bilinear form 
$Q: H_{\mathbf{R}} \times H_{\mathbf{R}} \to \mathbf{R}_Y$ 
which induces polarizations of Hodge structures on fibers.
A VMHS is said to be {\em graded polarizable} if 
each graded pieces $\text{Gr}^W_q(H_{\mathbf{Q}})$ 
is a polarizable variation of Hodge structures.

\begin{Thm}[\cite{Griffiths}]
Let $(H_{\mathbf{Z}}, F, Q)$ be a polarized variation of 
Hodge structures on a complex manifold $Y$.
Assume that $F^m(H_{\mathbf{C}}) \ne 0$ and 
$F^{m+1}(H_{\mathbf{C}}) = 0$.
Then the Gauss-Manin connection induces a connection on $F^m(H_{\mathbf{C}})$ 
whose curvature is positive semi-definite. 
\end{Thm}

\begin{proof}
Let $i: F^m(H_{\mathbf{C}}) \to H_{\mathbf{C}}$ and 
$p: H_{\mathbf{C}} \to H_{\mathbf{C}}/F^m(H_{\mathbf{C}})$ be
the natural homomorphisms.
We introduce a hermitian metric $h_H$ on $H_{\mathbf{C}}$
by $h_H(u,v) = Q(u,\bar v)$.
The {\em Gauss-Manin connection} $\nabla_H$ is flat, and $\Theta_H = 0$.

Let $h_F$ and $h_G$ be the induced hermitian metrics on the subsheaf 
$F^m(H_{\mathbf{C}})$ and the quotient sheaf 
$H_{\mathbf{C}}/F^m(H_{\mathbf{C}})$ respectively.
We calculate the curvature $\Theta_F$ using the second fundamental form
$b=p \circ \nabla_H \circ i$.
By the Griffiths transversality, the image of $b$ lies in
$\text{Gr}^W_{m-1}(H_{\mathbf{C}}) \otimes \Omega_Y^1$.

For holomorphic sections $u,v$ of $F^m(H_{\mathbf{C}})$, we have
\[
Q(\Theta_Fu,\bar v) = 
Q(\Theta_Hu,\bar v) - Q(b(u),\overline{b(v)})
\]
by Proposition~\ref{secondff}.
Since the non-degenerate bilinear form $Q$ has opposite signs
on $\text{Gr}_m^W(H_{\mathbf{C}})$ and 
$\text{Gr}_{m-1}^W(H_{\mathbf{C}})$, we obtain the
desired semipositivity.
\end{proof}

The following semi-simplicity theorem is very useful:

\begin{Thm}[\cite{Deligne-II}]\label{SS}
Let $Y$ be a complex manifold.
Then the category of polarized variations of Hodge structures on $Y$ is 
semi-simple in the sense that arbitrary injective homomorphism splits
over $\mathbf{Q}$.
\end{Thm}

\begin{proof}
Let $H_1 \to H$ be an injective homomorphism of variations of Hodge 
structures.
We take the orthogonal complement $H_{2,\mathbf{Q}}$ of $H_{1,\mathbf{Q}}$
in $H_{\mathbf{Q}}$ with respect to the polarization $Q$.
Then $H_2 = H \cap H_{2,\mathbf{Q}}$ is again a polarized variation of 
Hodge structures.
Indeed, if $u = \sum u^{p,q} \in H_{2,\mathbf{C}}$, then $Q(u^{p,q},v) = 
Q(u^{p,q},v^{q,p}) = Q(u,v^{q,p}) = 0$ for all 
$v = \sum v^{p,q} \in H_{1,\mathbf{C}}$,
hence $u^{p,q} \in H_{2,\mathbf{C}}$.
Moreover we have $H_{1,\mathbf{C}} \cap H_{2,\mathbf{C}} = 0$ because of the
Riemann-Hodge index theorem.
\end{proof}

\section{Weight filtration of the $\mathbf{Q}$-level complex}

Let $f: (X,B) \to (Y,C)$ be a weakly semistable toroidal model
of an algebraic fiber space.
Let $\rho_X: X^{\log} \to X$ and $\rho_Y: Y^{\log} \to Y$
be the associated log spaces with the induced continuous map
$f^{\log}: X^{\log} \to Y^{\log}$ which is 
topologically locally trivial.
Let $y \in Y$ be an arbitrary point, and $\bar y \in Y^{\log}$
a point above $y$.
If $y$ is contained in exactly $s$ irreducible components of $C$, 
then $\rho_Y^{-1}(y)$ is homeomorphic to $(S^1)^s$, and 
$\bar y$ is parametrized by angles $(\theta_1, \dots, \theta_s)$.

Let $E = f^{-1}(y)$ and $D = (f^{\log})^{-1}(\bar y)$ be the
fibers.
We shall put a weight filtration on a complex of sheaves 
$R(\rho_X)_*\mathbf{Z}_D$ on $E$ in this section.

Let $\{E_i\}$ be the set of irreducible components of $E$.
A {\em closed stratum} $E_I$ of $E$ is an irreducible component
of the intersection of some of the $E_i$.
Let $t = t(E_I) = \text{codim}_E E_I$.
We note that $t$ is not necessarily equal to the number of 
the $E_i$ which contain $E_I$.
Let $D_I = \rho_X^{-1}E_I \cap D$.
Then $\rho_X^{-1}E_I$ is homeomorphic to the product
$D_I \times \rho_Y^{-1}(y)$.
We denote $E^{[t]} = \coprod_{t(E_I)=t} E_I$ and 
$D^{[t]} = \coprod_{t(E_I)=t} D_I$.

Let $G_I$ be the union of all the strata which is properly contained
in $E_I$.
Then $(E_I,G_I)$ is a quasi-smooth toroidal variety.
Let $\rho_I: E_I^{\log} \to E_I$ be the associated log space.
We can see that $\rho_X: D_I \to E_I^{\log}$ is a $t$-times 
direct sum of oriented $S^1$ fiber bundles corresponding to the 
normal directions of $E_I$ in $E$. 

We recall the definition of a convolution of a complex of objects
in a triangulated category \cite{GM}.
Let
\[
a_0 \to a_1 \to \dots \to a_{n-1} \to a_n
\]
be a complex of objects.
If there exists a sequence of distinguished triangles
\[
b_{k-1} \to a_{k-1} \to b_k \to b_{k-1}[1] \\
\]
for $0 < k \le n$ with an isomorphism $b_n \to a_n$,
then $b_0$ is said to be a {\em convolution} of the complex.
A convolution may not exist and may not be unique if it exists.

We have a Mayor-Vietoris exact sequence
\[
0 \to \mathbf{Z}_D \to \mathbf{Z}_{D^{[0]}} \to 
\mathbf{Z}_{D^{[1]}} \to \mathbf{Z}_{D^{[2]}} \to \cdots
\]
In other words, $\mathbf{Z}_D$ is a convolution of a complex
\[
\mathbf{Z}_{D^{[0]}} \to 
\mathbf{Z}_{D^{[1]}} \to \mathbf{Z}_{D^{[2]}} \to \cdots
\]
Thus $R(\rho_X)_*\mathbf{Z}_D$ is a convolution of the 
following complex
\[
R(\rho_X)_*\mathbf{Z}_{D^{[0]}} \to 
R(\rho_X)_*\mathbf{Z}_{D^{[1]}} \to 
R(\rho_X)_*\mathbf{Z}_{D^{[2]}} \to \dots
\]

We define a weight filtration on $R(\rho_X)_*\mathbf{Z}_D$ 
as a convolution of canonical truncations:

\begin{Prop}
For any integer $q$, the following complex
\[
\tau_{\le q}(R(\rho_X)_*\mathbf{Z}_{D^{[0]}}) \to 
\tau_{\le q+1}(R(\rho_X)_*\mathbf{Z}_{D^{[1]}}) \to 
 \to \dots
\]
has a convolution $W_q(R(\rho_X)_*\mathbf{Z}_D)$,  
where $\tau$ denotes the canonical filtration, i.e.,
\[
H^p(\tau_{\le q+t}(R(\rho_X)_*\mathbf{Z}_{D^{[t]}}))
= \begin{cases}
H^p(R(\rho_X)_*\mathbf{Z}_{D^{[t]}}) &\text{ if } p \le q+t \\
0 &\text{ otherwise} \end{cases}
\]
which satisfies the following conditions:

(1) $W_q(R(\rho_X)_*\mathbf{Z}_D) \cong 0$ for sufficiently small $q$, and
$W_q(R(\rho_X)_*\mathbf{Z}_D) \cong R(\rho_X)_*\mathbf{Z}_D$ for 
sufficiently large $q$.

(2) There are distinguished triangles
\[
\begin{split}
&\text{Gr}^W_q(R(\rho_X)_*\mathbf{Z}_D)[-1] \to 
W_{q-1}(R(\rho_X)_*\mathbf{Z}_D) \to W_q(R(\rho_X)_*\mathbf{Z}_D) \\
&\to \text{Gr}^W_q(R(\rho_X)_*\mathbf{Z}_D)
\end{split}
\]
such that 
\[
\text{Gr}^W_q(R(\rho_X)_*\mathbf{Z}_D) \cong \bigoplus_{t \ge 0}
R^{q+t}(\rho_X)_*\mathbf{Z}_{D^{[t]}}[-q-2t].
\]
\end{Prop}

\begin{proof}
We denote $a_k = R(\rho_X)_*\mathbf{Z}_{D^{[k]}}$ and 
$a_k^q=\tau_{\le q+k}(R(\rho_X)_*\mathbf{Z}_{D^{[k]}})$.
Let $b_k$ and $b_k^q$ be the corresponding sequences of objects appearing 
in the process of convolutions.
We prove by the descending induction on $k$ that there is a unique morphism 
$a_{k-1}^q \to b_k^q$ which factors $a_{k-1}^q \to a_k^q$ and a 
morphism $b_k^q \to b_k$ such that they are 
compatible with the morphism $a_{k-1} \to b_k$.

\end{proof}

\begin{Cor}
The monodromy actions of the group 
$\pi_1(\rho_Y^{-1}(y),\bar y) \cong \mathbf{Z}^s$
on the cohomology groups 
$H^p(D,\mathbf{Z}_D)$ are unipotent for all $p$.
\end{Cor}

\begin{proof}
We have $H^p(D,\mathbf{Z}_D) = \mathbf{H}^p(E,R(\rho_X)_*\mathbf{Z}_D)$.
The monodromy actions on the graded pieces
$\mathbf{H}^p(E,\text{Gr}_q^W(R(\rho_X)_*\mathbf{Z}_D))$
are trivial, because $\rho_X^{-1}(E_I)$ are just homeomorphic to
the product spaces.
By the spectral sequence
\[
E_2^{p,q} = \mathbf{H}^{p+q}(E,\text{Gr}_{-p}^W(R(\rho_X)_*\mathbf{Z}_D))
\Rightarrow H^{p+q}(D,\mathbf{Z}_D)
\]
we conclude the proof.
\end{proof}

\begin{Expl}
Let $X = \text{Spec }\mathbf{C}[x_1,x_2]$, $B = \text{div}(x_1x_2)$,
$Y = \text{Spec }\mathbf{C}[y]$, $C = \text{div}(y)$, and define
$f: (X.B) \to (Y,C)$ by $f^*y=x_1x_2$.
Then $f^{\log}: X^{\log} \to Y^{\log}$ is given by
$(r_1,r_2,\theta_1,\theta_2) \mapsto (r_1r_2,\theta_1+\theta_2)$.
Let $\bar y = (0,\phi) \in Y^{\log}$.
Then $(f^{\log})^{-1} = D_1 \cup D_2$, where
$D_1=\{(r_1,0,\theta_1,\phi-\theta_1)\}$ and 
$D_2=\{(0,r_2,\phi-\theta_2,\theta_2)\}$.
There are homeomorphisms 
$D_1 \cong \mathbf{R}_{\ge 0} \times S^1$ and 
$D_2 \cong \mathbf{R}_{\ge 0} \times S^1$ given by
$(r_1,0,\theta_1,\phi-\theta_1) \mapsto (r_1,\theta_1)$ and 
$(0,r_2,\phi-\theta_2,\theta_2) \mapsto (r_2,\theta_2)$.
Then the gluing of $D_1$ and $D_2$ is given by
$(0,\theta) \mapsto (0,\phi - \theta)$ on $S^1$.
In other words, the gluing is twisted by the argument $\phi$ of $\bar y$.
If $\bar y$ goes around $S^1$, then the gluing is twisted correspondingly.
This is the action of the monodromy.
\end{Expl}

\section{Sheaves on log spaces}

Let $(X,B)$ be a quasi-smooth toroidal variety,
and $\rho: X^{\log} \to X$ the associated log space.
We shall define the \lq\lq structure sheaf'' 
$\mathcal{O}_{X^{\log}}$ and the De Rham complex
$\Omega^{\bullet}_{X^{\log}}$ in this section.

First we consider the case where $X = \mathbf{C}^n$ with 
coordinates $(x_1,\dots,x_n)$ and $B = \text{div}(x_1\dots x_r)$.
Then we define
\[
\mathcal{O}_{X^{\log}} = 
\sum_{k_1,\dots,k_r \in \mathbf{Z}_{\ge 0}}
\rho^{-1}(\mathcal{O}_X)\prod_{i=1}^r (\log x_i)^{k_i}
\]
where the symbols $\log x_i$ are regarded as holomorphic functions
over the open subset $\rho^{-1}(X \setminus \text{div}(x_i))$, 
while locally constant sections over the boundary $\rho^{-1}(\text{div}(x_i))$.
They are algebraically independent as long as they are symbols.
We note that the right hand side of the above definition
does not change if we replace the symbols $\log x_i$ by
the shifts $\log x_i + c_i$ for arbitrary constants $c_i \in \mathbf{C}$.
Thus $\mathcal{O}_{X^{\log}} \vert_{\rho^{-1}(X \setminus B)} = 
\mathcal{O}_X \vert_{X \setminus B}$, and 
the stalk $\mathcal{O}_{X^{\log},\bar x}$ 
at a point $\bar x \in \rho^{-1}(B)$ is isomorphic
to a polynomial ring $\mathcal{O}_{X,0}[t_{i_1},\dots,t_{i_l}]$, where 
the $t_i$ are independent variables corresponding to the symbols $\log x_i$
if $\rho(\bar x)$ is contained in exactly $l$ irreducible
components $B_{i_1}, \dots, B_{i_l}$ of $B$.

Next if $\pi: (X,B) \to (X,B)/G = (X',B')$ is a quotient 
of the above pair $(X,B)$ by a finite abelian group $G$, then
we define $\mathcal{O}_{(X')^{\log}} = 
(\pi^{\log}_*\mathcal{O}_{X^{\log}})^G$.
We note that the action of $G$ on $X^{\log}$ is free, hence
the stalks are isomorphic to polynomial rings for suitable 
number of variables. 

In the general case, the structure sheaves for an open covering of $X^{\log}$
are glued together to yield the structure sheaf of $X^{\log}$ because
they coincide with the usual structure sheaves when restricted to 
$X \setminus B$.

The sheaf of differentials is defined by the following formula:
\[
\Omega^p_{X^{\log}} = \rho^{-1}(\Omega^p_X(\log B))
\otimes_{\rho^{-1}(\mathcal{O}_X)} \mathcal{O}_{X^{\log}}.
\]
The differential $d: \Omega^p_{X^{\log}}
\to \Omega^{p+1}_{X^{\log}}$ is defined by the rule
$d(\log x) = dx/x$.
It follows that the De Rham 
complex $\Omega^{\bullet}_{X^{\log}}$ does not have
higher cohomologies:

\begin{Prop}[Poincar\'e lemma \cite{AF}]
\[
\begin{split}
&\mathbf{C}_{X^{\log}}
\cong \Omega^{\bullet}_{X^{\log}} \\
&R\rho_*\Omega^p_{X^{\log}} \cong \Omega^p_X(\log B) \\
&R\rho_*\mathbf{C}_{X^{\log}} \cong 
\Omega^{\bullet}_X(\log B)
\end{split}
\]
\end{Prop}

Let $f: (X,B) \to (Y,B)$ be a well prepared toroidal morphism.
Then the sheaves of relative differential forms are defined similarly:
\[
\Omega^p_{X^{\log}/Y^{\log}} = 
\rho^{-1}(\Omega^p_{X/Y}(\log))
\otimes_{\rho^{-1}(\mathcal{O}_X)} \mathcal{O}_{X^{\log}}.
\]
In the following, 
we denote $\rho'=(\rho_X,f^{\log}): X^{\log}  
\to X \times_Y Y^{\log}$ with the projections
$q_1: X \times_Y Y^{\log} \to X$ and 
$q_2: X \times_Y Y^{\log} \to Y^{\log}$.

\begin{Prop}[Poincar\'e lemma 2 \cite{AF}]
\[
\begin{split}
&(f^{\log})^{-1}(\mathcal{O}_{Y^{\log}})
\cong \Omega^{\bullet}_{X^{\log}/Y^{\log}} \\
&R\rho'_*\Omega^p_{X^{\log}/Y^{\log}} 
\cong q_1^{-1}(\Omega^p_{X/Y}(\log)) 
\otimes_{q_2^{-1}\rho_Y^{-1}(\mathcal{O}_Y)}
q_2^{-1}(\mathcal{O}_{Y^{\log}}) \\
&R\rho'_*(f^{\log})^{-1}(\mathcal{O}_{Y^{\log}}) 
\cong q_1^{-1}(\Omega^{\bullet}_{X/Y}(\log)) 
\otimes_{q_2^{-1}\rho_Y^{-1}(\mathcal{O}_Y)}
q_2^{-1}(\mathcal{O}_{Y^{\log}})
\end{split} 
\]
\end{Prop}

Now we define sheaves on the fiber $D$ and its strata $D_I$.
The restriction of $\rho_X$ to $D$ is denoted by $\rho_D$.
Let $y_i$ be the local coordinates at $y \in Y$ which define 
irreducible components of $C$ passing through $y$.
Then the symbols $\log y_i$ are replaced by $0$ on $D$.
we write this fact by $\sim$ in the following definition:
\[
\begin{split}
&\mathcal{O}_D = (\mathcal{O}_{X^{\log}} 
\otimes_{\rho_X^{-1}(\mathcal{O}_X)} 
\rho_D^{-1}(\mathcal{O}_E))/\sim \\
&\mathcal{O}_{D_I} = (\mathcal{O}_{X^{\log}} 
\otimes_{\rho_X^{-1}(\mathcal{O}_X)} 
\rho_D^{-1}(\mathcal{O}_{E_I}))/\sim \\
&\Omega^p_D = \rho_D^{-1}\Omega^p_{E/\mathbf{C}}(\log) 
\otimes_{\rho_D^{-1}(\mathcal{O}_E)} \mathcal{O}_D 
= \Omega^p_{X^{\log}/Y^{\log}} \otimes_{\mathcal{O}_{X^{\log}}}
\mathcal{O}_D \\
&\Omega^p_{D_I} = \rho_D^{-1}\Omega^p_{E/\mathbf{C}}(\log) 
\otimes_{\rho_D^{-1}(\mathcal{O}_E)} \mathcal{O}_{D_I} 
= \Omega^p_{X^{\log}/Y^{\log}} \otimes_{\mathcal{O}_{X^{\log}}}
\mathcal{O}_{D_I}
\end{split}
\]
Let 
\[
\Omega^p_{E/\mathbf{C}}(\log) = \Omega^p_{X/Y}(\log) \otimes \mathcal{O}_E.
\]

\begin{Cor}[Poincar\'e lemma 3 \cite{AF}]
\[
\begin{split}
&\mathbf{C}_D \cong \Omega^{\bullet}_D, \qquad 
\mathbf{C}_{D_I} \cong \Omega^{\bullet}_{D_I} \\
&R\rho_{D*}\Omega^p_D \cong \Omega^p_{E/\mathbf{C}}(\log), \qquad 
R\rho_{D*}\Omega^p_{D_I} \cong \Omega^p_{E/\mathbf{C}}(\log)
\otimes_{\mathcal{O}_E} \mathcal{O}_{E_I} \\
&R\rho_{D*}\mathbf{C}_D \cong 
\Omega^{\bullet}_{E/\mathbf{C}}(\log), \qquad  
R\rho_{D*}\mathbf{C}_{D_I} \cong 
\Omega^{\bullet}_{E/\mathbf{C}}(\log)
\otimes_{\mathcal{O}_E} \mathcal{O}_{E_I}
\end{split}
\]
\end{Cor}

The formation of the {\em canonical extension} is easily understood 
when we considers the structure sheaf of the log space as follows.
Let $(Y,C)$ be a smooth toroidal variety, $H_{\mathbf{Z}}$ 
a locally constant sheaf of free abelian groups
on $Y \setminus C$, 
and let $H = H_{\mathbf{Z}} \otimes \mathcal{O}_{Y \setminus C}$ be 
the associated locally free sheaf on the open part $Y \setminus C$.
Let $y \in C$ be an arbitrary point on the boundary, 
$x_j$ the local coordinates corresponding to the local branches of $C$ 
passing through $y$, and $T_j$ the monodromy transformations of 
the local system $H_{\mathbf{Z}}$ around the branches of $C$ 
corresponding to the $x_j$.
We assume that the local monodromies $T_j$ are unipotent for any $y$.
Then the {\em canonical extension} $\tilde H$ of $H$ is defined as 
a locally free sheaf
on $Y$ generated locally around $y$ by the following type of local sections
of $H$:
\[
\text{exp}(-\frac 1{2\pi i}\sum_j \log T_j \log x_i)v
\]
where the $v$ are multivalued flat sections of $H_{\mathbf{Z}}$ and the 
power series for $\text{exp}$ is finite because the $\log T_j$ are 
nilpotent.
We note that the above expressions are single valued sections of $H$ 
because the 
monodromy transformations are cancelled.

\begin{Prop}
Let $H_{\mathbf{Z}}^{\log}$ be the local system on the log space $Y^{\log}$
obtained by extensing the local system $H_{\mathbf{Z}}$.
Then 
\[
H_{\mathbf{Z}}^{\log} \otimes \mathcal{O}_{Y^{\log}}
\cong \rho_Y^{-1}\tilde H \otimes_{\rho_Y^{-1}\mathcal{O}_Y} 
\mathcal{O}_{Y^{\log}}.
\]
\end{Prop}

\section{Existence of a cohomological mixed Hodge complex}

We shall prove the existence of a 
cohomological mixed Hodge complex
on a singular fiber of a weakly semistable algebraic fiber space.

\begin{Thm}
Let $f: (X,B) \to (Y,C)$ be a weakly semistable algebraic fiber space,
i.e., a projective surjective toroidal morphism from a
quasi-smooth toroidal variety to a smooth toroidal variety having
connected, reduced and equi-dimensional geometric fibers.
Let $f^{\log}: X^{\log} \to Y^{\log}$ be the induced continuous map between
the associated log spaces $\rho_X: X^{\log} \to X$ and 
$\rho_Y: Y^{\log} \to Y$.
Let $y \in Y$ be an arbitrary point, $\bar y \in \rho_Y^{-1}(y)$, 
$E = f^{-1}(y)$ and $D = (f^{\log})^{-1}(\bar y)$.
Denote $\rho_D = \rho_X \vert_D$.
Then the following data is a cohomological mixed Hodge complex on $E$.

(1) $H_{\mathbf{Z}} = R\rho_{D*}\mathbf{Z}_D$.

(2) $H_{\mathbf{C}} = \Omega^{\bullet}_{E/\mathbf{C}}(\log)$.

(3) A weight filtration $W_q(R\rho_{D*}\mathbf{Q}_D)$ 
on $H_{\mathbf{Q}}$ is defined as a convolution of the 
following complex of objects:
\[
\tau_{\le q}(R\rho_{D*}\mathbf{Q}_{D^{[0]}}) \to   
\tau_{\le q+1}(R\rho_{D*}\mathbf{Q}_{D^{[1]}}) \to   
\tau_{\le q+2}(R\rho_{D*}\mathbf{Q}_{D^{[2]}}) \to \dots  
\]
where $\tau$ denotes the canonical filtration.

(4) A weight filtration $W_q(\Omega^{\bullet}_{E/\mathbf{C}}(\log))$ 
on $H_{\mathbf{C}}$ is defined as a convolution of the 
following complex of objects:
\[
\begin{split}
&W_q(\Omega^{\bullet}_{E/\mathbf{C}}(\log) 
\otimes_{\mathcal{O}_E} \mathcal{O}_{E^{[0]}}) \to 
W_{q+1}(\Omega^{\bullet}_{E/\mathbf{C}}(\log) 
\otimes_{\mathcal{O}_E} \mathcal{O}_{E^{[1]}}) \\
&\to W_{q+2}(\Omega^{\bullet}_{E/\mathbf{C}}(\log) 
\otimes_{\mathcal{O}_E} \mathcal{O}_{E^{[2]}}) \to \dots
\end{split}
\]
where $W$ denotes the filtration determined by the order of 
log poles.

(5) A Hodge filtration on $H_{\mathbf{C}}$ is the stupid filtration:
\[
F^p(\Omega^{\bullet}_{E/\mathbf{C}}(\log))
= \Omega^{\ge p}_{E/\mathbf{C}}(\log).
\]
\end{Thm}

\begin{proof}
By the Poincar\'e lemma, we have 
$H_{\mathbf{Z}} \otimes \mathbf{C} \cong H_{\mathbf{C}}$, and
\[
\begin{split}
&\text{Gr}^W_q(H_{\mathbf{Q}}) \cong \bigoplus_{t \ge 0}
R^{q+t}\rho_{D*}\mathbf{Q}_{D^{[t]}}[-q-2t] \\
&\text{Gr}^W_q(H_{\mathbf{C}}) \cong \bigoplus_{t \ge 0}
\text{Gr}^W_{q+t}(\Omega^{\bullet}_{E/\mathbf{C}}(\log)
\otimes_{\mathcal{O}_E} \mathcal{O}_{E^{[t]}})[-t] \\
&\text{Gr}^W_{q+t}(\Omega^{\bullet}_{E/\mathbf{C}}(\log)
\otimes_{\mathcal{O}_E} \mathcal{O}_{E^{[t]}})[-t]
\cong R^{q+t}\rho_{D*}\mathbf{C}_{D_I}[-q-2t].
\end{split}
\]

We shall prove that 
\[
(\text{Gr}^W_{q+t}(\Omega^{\bullet}_{E/\mathbf{C}}(\log)
\otimes_{\mathcal{O}_E} \mathcal{O}_{E^{[t]}})[-t],F)
\]
is a cohomological mixed Hodge complex of weight $q$ on $E$.

We have an exact sequence
\[
0 \to \Omega^1_{E_I}(\log G_I) \to 
\Omega^1_{E/\mathbf{C}}(\log) \otimes_{\mathcal{O}_E} \mathcal{O}_{E_I}
\to \mathcal{O}_{E_I}^t \to 0
\]
where the last arrow is given by the residue homomorphisms
along the normal directions of $E_I$ in $E$. 

Let $G_J$ be a closed stratum on $E_I$, an irreducible component
of the intersection of some of the irreducible components of $G_I$.
Let $s = \text{codim}_{E_I} G_J$.
When we derive a differential form on $G_J$ from a section of 
$\Omega^{\bullet}_{E/\mathbf{C}}(\log) \otimes_{\mathcal{O}_E} 
\mathcal{O}_{E_I}$,
we take residues $s$ times in the normal 
directions of $G_J$ in $E_I$ and $t'$ times in the directions
of $\mathcal{O}_{E_I}^t$ of the above exact sequence, where
we have $0 \le t' \le t$.
In order to obtain the graded piece of degree $q+t$, we have to take
residues $q+t$ times.
Thus we have $q+t=s+t'$.
The degree of the differential form drops also by $q+t$.
Therefore we obtain
\[
(\text{Gr}^W_{q+t}(\Omega^{\bullet}_{E/\mathbf{C}}(\log)
\otimes_{\mathcal{O}_E} \mathcal{O}_{E^{[t]}})[-t],F)
\cong (\binom{t}{t'}\Omega^{\bullet}_{G^{[s]}}[-t-q-t],F(-q-t)).
\]
Since
\[
-t-q-t + 2(q+t) = q
\]
the last term is a cohomological Hodge complex of weight $q$.
\end{proof}

The formula obtained in the above proof $q = s + t' - t$
says that the shift of the weight along the degeneration of fibers 
can be in both positive and negative directions, while the shift of the
cohomology degree $s+t'+t$ is in one direction.

The following is an immediate corollary:

\begin{Cor}
Assume the conditions of the theorem. 
Then the following hold:

(1) The spectral sequence associated to the Hodge filtration
\[
{}_FE_1^{p,q} = H^q(E,\Omega^p_{E/\mathbf{C}}(\log)) \Rightarrow
H^{p+q}(D, \mathbf{C})
\]
degenerates at $E_1$.

(2) The spectral sequence associated to the weight filtration
\[
{}_WE_1^{p,q} = H^{p+q}(E,\text{Gr}^W_{-p}(R\rho_{D*}\mathbf{Q}_D)) 
\Rightarrow H^{p+q}(D,\mathbf{Q})
\]
degenerates at $E_2$.
\end{Cor}

The upper semi-continuity theorem yields the local freeness theorem:

\begin{Cor}
Assume the conditions of the theorem. 
Then the following hold:

(1) The higher direct image sheaves $R^qf_*\Omega_{X/Y}^p(\log)$ on $Y$ 
are locally free for all $p,q$.

(2) Let $\tilde H^k$ be the canonical extension of the 
higher direct image sheaf
$R^kf^o_*\mathbf{C}_{X \setminus B} 
\otimes_{\mathbf{C}_{Y \setminus C}} \mathcal{O}_{Y \setminus C}$ 
for any integer $k$.
Then there is an isomorphism 
\[
\tilde H^k \cong R^kf_*\Omega_{X/Y}^{\bullet}(\log).
\]
Moreover 
\[
F^p(\tilde H^k) \cong R^kf_*\Omega_{X/Y}^{\ge p}(\log)
\]
gives an 
increasing filtration by locally free subsheaves of the canonical extension.
\end{Cor}

\begin{proof}
(1) The rank of the cohomology groups 
$\bigoplus_{p+q=k} H^q(E,\Omega_{E/\mathbf{C}}^p(\log))$ for $E = f^{-1}(y)$ 
is independent of $y \in Y$.
Hence the assertion follows from the upper semi-continuity theorem.

(2) We have
\[
\begin{split}
&\rho_Y^{-1}\tilde H^k \otimes_{\rho_Y^{-1}\mathcal{O}_Y} 
\mathcal{O}_{Y^{\log}} 
\cong R^kf^{\log}_*\mathbf{C}_{X^{\log}}
\otimes \mathcal{O}_{Y^{\log}} \\
&\cong \rho_Y^{-1}R^kf_*\Omega_{X/Y}^{\bullet}(\log)
\otimes_{\rho_Y^{-1}\mathcal{O}_Y} \mathcal{O}_{Y^{\log}} \\
&\rho_Y^{-1}F^p(\tilde H^k) \otimes_{\rho_Y^{-1}\mathcal{O}_Y} 
\mathcal{O}_{Y^{\log}} 
\cong F^p(R^kf^{\log}_*\mathbf{C}_{X^{\log}}
\otimes \mathcal{O}_{Y^{\log}}) \\
&\cong \rho_Y^{-1}F^p(R^kf_*\Omega_{X/Y}^{\bullet}(\log))
\otimes_{\rho_Y^{-1}\mathcal{O}_Y} \mathcal{O}_{Y^{\log}}
\end{split}
\]
hence the result.
\end{proof}

The above assertion can be derived from the nilpotent orbit theorem
(\cite{Schmid}).
Our proof is more geometric.

\section{Semipositivity theorem}

A locally free sheaf $F$ on a proper algebraic variety $Y$ is said to be
{\em numerically semipositive} if the tautological invertible sheaf 
$\mathcal{O}_{\mathbf{P}(F)}(1)$ on the associated projective space bundle 
is nef.
It is equivalent to saying that, for all morphisms $\phi: \Gamma \to Y$ from a 
smooth projective curve $\Gamma$ and for all quotient invertible sheaves 
$G$ of the inverse images $\phi^*F$, the inequalities 
$\text{deg}_{\Gamma}(G) \ge 0$ hold.

\begin{Thm}\label{SP}
Let $f: (X,B) \to (Y,C)$ be a well prepared algebraic fiber space of 
relative dimension $n$, i.e., a projective surjective toroidal morphism from a
quasi-smooth toroidal variety to a smooth toroidal variety having 
connected and equi-dimensional geometric fibers.
Then the sheaves $f_*\Omega_{X/Y}^k(\log)$ 
and $R^kf_*\Omega_{X/Y}^n(\log)$ are numerically 
semipositive locally free sheaves 
for all non-negative integers $k$.
\end{Thm}

If $f: (X,B) \to (Y,C)$ is weakly semistable, 
i.e., if all the geometric fibers are reduced, then
$\Omega_{X/Y}^n(\log) = \omega_{X/Y}(B^h)$
is the relative dualising sheaf twisted by the horizontal part of the 
boundary $B$.
But the both sides are different in the general case.

\begin{proof}
At first, we consider the case where $f$ is weakly semistable and there is no
horizontal component of the boundary, i.e., $B = f^{-1}(C)$.
The following proof is essentially in \cite{AB}.

Let $H_t = R^tf_*\Omega_{X/Y}^{\bullet}(\log)$. 
It is a locally free sheaf on $Y$ which is the canonical 
extension of a variation of Hodge structures on $Y \setminus C$.
We denote $F_{n-k} = f_*\Omega_{X/Y}^{n-k}(\log)$ and 
$F_{n+k} = R^kf_*\Omega_{X/Y}^n(\log)$.
They are locally free subsheaves of the $H_t$ which coincide with
the canonical extensions of the highest Hodge filtration 
for $t=n-k$ or $t=n+k$.

Let $L$ be the cohomology class of an ample divisor on $X$.
By the Lefschetz theorem, we have an isomorphism
$L^k: H_{n-k} \to H_{n+k}$ which sends $F_{n-k}$ to $F_{n+k}$.
The primitive part 
\[
{}_0H_{n-k} = \text{Ker}(L^{k+1}: H_{n-k} \to H_{n+k+2})
\]
is the canonical extension of a polarized variation of Hodge structures on
$Y \setminus C$.
We note that $F_{n-k}$ is contained in ${}_0H_{n-k}$.

We shall prove that $F_t$ for $t = n - k$ is semipositive. 
Let $\phi: \Gamma \to Y$ and $G$ as above.
We assume first that the image of $\phi$ is general in the sense that
$\Gamma^o = \phi^{-1}(Y \setminus C) \ne \emptyset$.
The bilinear form $Q(u,v) = \int u \cup v \cup L^k$ induces a 
positive definite hermitian metric $h_F$ on $F_t \vert_{Y \setminus C}$.
Let $h_G$ be the induced hermitian metric on $G \vert_{\Gamma^o}$.
By the Griffiths semipositivity, the curvature form of $h_Q$ is 
semi-positive.

We regard $h_G$ as a hermitian metric on the whole space $\Gamma$ 
with singularities along the boundary $\Gamma - \Gamma^o$, namely a 
{\em singular hermitian metric}.
It can be written locally near a point $y \in \Gamma - \Gamma^o$ in a form
$h_G = e^gh_0$ for a $C^{\infty}$ hermitian metric $h_0$ 
and a locally integrable function $g$.
Then we can express the curvature current $\Theta_G$ of $h_G$ in a form
$\Theta_G = \Theta_0 + \bar{\partial} \partial g$, where $\Theta_0$ is 
the curvature from of $h_0$.
The presentation of the canonical extension implies that the function
$g$ is the logarithm of a function which is expressed by 
$C^{\infty}$ functions and logarithmic functions.
Therefore the boundary contribution due to the singularities of $h_G$ 
vanishes, and we have 
\[
\text{deg}_{\Gamma}(G) = \int_{\Gamma_0} \frac i{2\pi} \Theta_G \ge 0.
\]

Next we consider the case where $\phi(\Gamma)$ is contained in
the boundary $C$.
Let $C_I$ be the irreducible component of the intersection of 
some of the irreducible components of $C$ such that 
the image of $\phi$ is contained in $C_I$ and general in the sense that
$\Gamma^o = \phi^{-1}(C_I^o) \ne \emptyset$
where we set $C_I^o = C_I \setminus \bar C^{[s+1]}$ 
for the image $\bar C^{[s+1]}$ 
of $C^{[s+1]}$ in $C_I$ with $s = \text{codim}_YC_I$.

For $y \in C_I^o$ and $\bar y \in \rho_Y^{-1}(y)$, 
we use the notation as in the previous sections 
such as $E = f^{-1}(y)$ and $D = (f^{\log})^{-1}(\bar y)$.
Moreover we denote $C_I^{o\log} = \rho_Y^{-1}(C_I^o)$, 
$E_I^o = f^{-1}(C_I^o)$ and 
$D_I = \rho_X^{-1}(E_I^o) = f^{\log -1}(C_I^{o\log})$.
The mixed Hodge structure defined previously on $H^t(D, \mathbf{Z})$
varies continuously when $\bar y$ moves, and becomes a continuous
variation of mixed Hodge structures
$H_{\mathbf{Z}}^{\log} = R^tf^{\log}_*\mathbf{Z}_{D_I}$ 
on $C_I^{o \log}$.

Since the weight filtration $W$ is flat along this variation, 
there exists a 
filtration, denoted by $W$ again, on $H_{t,I}^o = H_t \otimes_{\mathcal{O}_Y} 
\mathcal{O}_{C_I^o}$ by locally free subsheaves on $C_I^o$ such that 
\[
W_q(H_{\mathbf{Z}}^{\log}) \otimes \mathcal{O}_{C_I^{o \log}}
\cong \rho_Y^{-1}W_q(H_{t,I}^o) \otimes_{\rho_{\mathcal{O}_{C_I^o}}} 
\mathcal{O}_{C_I^{o \log}}.
\]
The Hodge filtration $F$ on $H_{\mathbf{Z}}^{\log}$ 
also comes from downstairs:
\[
F^p(H_{\mathbf{Z}}^{\log} \otimes \mathcal{O}_{C_I^{o \log}})
\cong \rho_Y^{-1}F^p(H_{t,I}^o) \otimes_{\rho_{\mathcal{O}_{C_I^o}}} 
\mathcal{O}_{C_I^{o \log}}.
\]
We denote by $F_{t,I}^o = F^t(H_{t,I}^o)$ the highest Hodge part,
which coincides with 
$F_t \otimes_{\mathcal{O}_Y} \mathcal{O}_{C_I^o}$.
The canonical extension $H_{t,I}$ (resp. $F_{t,I}$) of $H_{t,I}^o$ 
(resp. $F_{t,I}^o$) across the 
boundary $\bar C^{[s+1]}$ coincides with 
$H_t \otimes_{\mathcal{O}_Y} \mathcal{O}_{C_I}$ 
(resp. $F_t \otimes_{\mathcal{O}_Y} \mathcal{O}_{C_I}$). 

Let $q$ be the smallest integer such that 
the surjective homomorphism $\phi^*F_t \to G$
induces a non-zero homomorphism $W_q(F_{t,I}) \to G$.
Then we have a non-zero homomorphism 
$\text{Gr}^W_q(F_{t,I}) \to G$.
The graded piece $\text{Gr}^W_q(H_{t,I})$ is a direct sum of the 
canonical extensions of variations of (pure) Hodge structures on $C_I$
defined by irreducible components of the intersection of some of 
the irreducible components of $E_I = f^{-1}(C_I)$.
Therefore we infer that $\text{deg}_{\Gamma}(G) \ge 0$ by the 
first part of this proof.

Now we consider the generalization to the case of 
well prepared algebraic fiber spaces without horizontal boundary components.
Let $\pi_Y: Y' \to Y$ a finite surjective Galois base change morphism
such that the induced algebraic fiber space $f': X' \to Y'$ is
weakly semistable.
We denote by $\pi_X: X' \to X$ the induced finite morphism.
We note that the ramification of $\pi_Y$ occurs not only along the
boundary divisor $C$, though we have $C' = \pi^{-1}C$.
Our assertion follows from the following two lammas.

\begin{Lem}
(1) $\pi_X^*\Omega^p_{X/Y}(\log) \cong \Omega^p_{X'/Y'}(\log)$ for all $p$.

(2) $R^qf_*\Omega^p_{X/Y}(\log)$ is locally free for all $p,q$.

(3) $\pi_Y^*R^qf_*\Omega^p_{X/Y}(\log) \cong R^qf'_*\Omega^p_{X'/Y'}(\log)$
for all $p,q$.
\end{Lem}

\begin{proof}
(1) The isomorphism is the pull-back homomorphism.
We check that it is bijective.
If $x \in X \setminus f^{-1}(C)$, then it is bijective 
in a neighbborhood of $x$ because $f$ is smooth there.
If $f(x) \in C$ but $f(x)$ is not contained in any irreducible component
of the discriminant locus of $\pi_Y$ other than those of $C$, 
then we have
$\pi_X^*\Omega^p_X(\log B) \cong \Omega^p_{X'}(\log B')$ and 
$\pi_Y^*\Omega^p_Y(\log C) \cong \Omega^p_{Y'}(\log C')$, hence an isomorphism
in a neighbborhood of $x$.
Therefore our homomorphism is bijective outside a codimension $2$ subset, 
and we are done.

(2) Since $(\pi_{X*}\pi_X^*\Omega^p_{X/Y}(\log))^G \cong \Omega^p_{X/Y}(\log)$,
we have 
\[
R^qf_*\Omega^p_{X/Y}(\log) \cong 
(\pi_{Y*}R^qf'_*\Omega^p_{X'/Y'}(\log))^G.
\]
Since $\pi_Y$ is flat, we have our assertion.

(3) This follows from the local freeness theorem and 
the upper semi-continuity theorem.
\end{proof}

\begin{Lem}
Let $\pi: Y' \to Y$ be a finite surjective 
morphism of proper varieties, and let $F$ be a locally free sheaf on $Y$.
Then $F$ is numerically semipositive if and only 
if $\pi^*F$ is numerically semipositive.
\end{Lem}

\begin{proof}
The assertion is reduced to the case where the rank of $F$ is equal to $1$, 
and the latter is clear.
\end{proof}

We postpone the proof of the theorem in the case where there 
are horizontal boundary components to the end of the next section.
\end{proof}

\section{Weight spectral sequences}

We use the following notation.
$f: (X.B) \to (Y,C)$ is the given well prepared algebraic fiber space.
We write $B = B^h + B^v$ for $B^v = f^{-1}(C)$ ($h$ stands for 
horizontal and $v$ the vertical).
We denote by $B^h_I$ an irreducible component of the intersection of 
some of the irreducible components of $B^h$, where $I$ stands for 
a set of indices which determines $B^h_I$ in a similar way as in the previous 
sections.
Such $B^h_I$ is called a {\em stratum} of $X$ in this proof.
For a fixed stratum $B^h_I$, we denote by $G_I = B^v \cap B^h_I$.
Then the induced morphism $f_I: (B^h_I,G_I) \to (Y,C)$
is a well prepared algebraic fiber spaces
without horizontal boundary components.
In particular, we have $(B^h_{\emptyset},G_{\emptyset}) = (X,B^v)$.

Let $\rho_{B^h_I}: B^{h\log}_I \to (B^h_I,G_I)$
be the real oriented blowing-up, 
and let $f^{\log}_I: B^{h\log}_I \to Y^{\log}$ be the induced morphism
of topological spaces.
The induced morphism $\sigma: X^{\log} \to B^{h\log}_{\emptyset}$ is the
partial real oriented blowing-up along the horizontal boundary component 
such that $f^{\log} = f^{\log}_{\emptyset} \circ \sigma$.

We define a weight filtration $W$ on a higher direct image 
$R\sigma_*\mathbf{Q}_{X^{\log}}$ on $B^{h\log}_{\emptyset}$ 
as the canonical filtration $\tau$ as before.
Then we have
\[
\text{Gr}^W_q(R\sigma_*\mathbf{Q}_{X^{\log}})
\cong \bigoplus_{\text{codim }B^h_I = q} \mathbf{Q}_{B^{h\log}_I}[-q]
\]
where the $B^{h\log}_I$ are regarded as closed subspaces 
of $B^{h\log}_{\emptyset}$.
There is an induced weight filtration, denoted again by $W$, 
on the higher direct image 
$Rf^{\log}_*\mathbf{Q}_{X^{\log}}$ on $Y^{\log}$ such that
\[
\text{Gr}^W_q(Rf^{\log}_*\mathbf{Q}_{X^{\log}})
\cong \bigoplus_{\text{codim }B^h_I = q} 
Rf^{\log}_*\mathbf{Q}_{B^{h\log}_I}[-q].
\]
In other words, we have distinguished triangles
\[
\begin{split}
&W_{q-1}(Rf^{\log}_*\mathbf{Q}_{X^{\log}}) \to 
W_q(Rf^{\log}_*\mathbf{Q}_{X^{\log}}) \\
&\to 
\text{Gr}^W_q(Rf^{\log}_*\mathbf{Q}_{X^{\log}}) \to
W_{q-1}(Rf^{\log}_*\mathbf{Q}_{X^{\log}})[1]
\end{split}
\]
so that $Rf^{\log}_*\mathbf{Q}_{X^{\log}}$ is a convolution of a 
complex of objects
\[
\begin{split}
&\cdots \to \bigoplus_{\text{codim }B^h_I = q} 
Rf^{\log}_*\mathbf{Q}_{B^{h\log}_I}[-2q] 
\to \bigoplus_{\text{codim }B^h_I = q-1} 
Rf^{\log}_*\mathbf{Q}_{B^{h\log}_I}[-2q+2]
\to \\
&\dots \to 
\bigoplus_{\text{codim }B^h_I = 1} 
Rf^{\log}_*\mathbf{Q}_{B^{h\log}_I}[-2]
\to Rf^{\log}_*\mathbf{Q}_{B^{h\log}_{\emptyset}}.
\end{split}
\]

We have a decomposition theorem for the constant sheaf when there is no horizontal components:

\begin{Thm}
Let $f: (X.B) \to (Y,C)$ be a well prepared algebraic fiber space,
and $f^{\log}: X^{\log} \to Y^{\log}$ the associated map of log spaces.
Assume that there is no horizontal component of $B$, i.e., $B = f^{-1}(C)$.
Then 
\[
Rf^{\log}_*\mathbf{Q}_{X^{\log}}
\cong \bigoplus_i R^if^{\log}_*\mathbf{Q}_{X^{\log}}[-i].
\]
\end{Thm}

\begin{proof}
Let $L \in H^2(X,\mathbf{Z})$ be the class of an ample divisor
on $X$, and $y \in Y \setminus C$ a point.
Then the fiber $f^{-1}(y)$ is a smooth projective variety and the
Lefschetz decomposition theorem holds.
Therefore we have a decomposition theorem for the fiber
by Theorem~\ref{LD}.
Since our sheaves are locally constant on the whole log space
$Y^{\log}$, the Lefschetz decomposition theorem extends to singular fibers,
and we are done.
\end{proof}

By the above theorem, we have
\[
Rf^{\log}_*\mathbf{Q}_{B^{h\log}_I}
\cong \bigoplus_p R^pf^{\log}_*\mathbf{Q}_{B^{h\log}_I}[-p].
\]
With respect to this decomposition, 
the boundary morphisms of the above complex are the sums of 
the Gysin homomorphisms
\[
R^pf^{\log}_*\mathbf{Q}_{B^{h\log}_I} \to 
R^{p+2}f^{\log}_*\mathbf{Q}_{B^{h\log}_J}
\]
for immersions $B^h_I \subset B^h_J$ of codimension $1$.

There is a spectral sequence associated to the weight filtration
\begin{equation}\label{horizontal1}
\begin{split}
&E_1^{p,q} = H^{p+q}(\text{Gr}^W_{-p}(Rf^{\log}_*\mathbf{Q}_{X^{\log}})
= \bigoplus_{\text{codim }B^h_I = -p} 
R^{2p+q}f^{\log}_*\mathbf{Q}_{B^{h\log}_I} \\
&\Rightarrow R^{p+q}f^{\log}_*\mathbf{Q}_{X^{\log}}
\end{split}
\end{equation}
where the differentials $d_1^{p,q}$ are nothing but the Gysin
homomorphisms above.
It degenerates at $E_2$ by \cite{Deligne-II}; 
we have $d_r^{p,q}=0$ for all $r \ge 2$.

\vskip 1pc

Next we define a weight filtration on 
$\bigoplus_k \Omega^k_{X/Y}(\log)[-k]$
by the order of log poles along $B^h$ so that we have
\[
\text{Gr}^W_q(\bigoplus_k \Omega^k_{X/Y}(\log)[-k])
\cong \bigoplus_k \bigoplus_{\text{codim }B^h_I = q} 
\Omega^k_{B^h_I/Y}(\log)[-k-q]
\]
where the $B^h_I$ are regarded as closed subspaces 
of $X$ whose boundaries have only vertical components.
We note that the isomorphism does not hold for individual $k$.
There is an induced weight filtration, denoted again by $W$, 
on the higher direct image 
$\bigoplus_k Rf_*\Omega^k_{X/Y}(\log)[-k]$ on $Y$ such that
\[
\text{Gr}^W_q(\bigoplus_k Rf_*\Omega^k_{X/Y}(\log)[-k])
\cong \bigoplus_k \bigoplus_{\text{codim }B^h_I = q} 
Rf_*\Omega^k_{B^h_I/Y}(\log)[-k-q].
\]
Thus $\bigoplus_k Rf_*\Omega^k_{X/Y}(\log)[-k]$ is a convolution of a 
complex of objects
\[
\begin{split}
&\cdots \to \bigoplus_k \bigoplus_{\text{codim }B^h_I = q} 
Rf_*\Omega^k_{B^h_I/Y}(\log)[-k-2q] \\
&\to \bigoplus_k \bigoplus_{\text{codim }B^h_I = q-1} 
Rf_*\Omega^k_{B^h_I/Y}(\log)[-k-2q+2]
\to \\
&\dots \to 
\bigoplus_k \bigoplus_{\text{codim }B^h_I = 1} 
Rf_*\Omega^k_{B^h_I/Y}(\log)[-k-2] \\
&\to \bigoplus_k Rf_*\Omega^k_{B^h_{\emptyset}/Y}(\log)[-k].
\end{split}
\]

Now we prove the decomposition theorem for the sheaves of logarithmic 
differential forms when there is no horizontal components:

\begin{Cor}
Let $f: (X.B) \to (Y,C)$ be a well prepared algebraic fiber space.
Assume that there is no horizontal component of $B$, i.e., $B = f^{-1}(C)$.
Then 
\[
Rf_*\Omega^k_{X/Y}(\log)
\cong \bigoplus_i R^if_*\Omega^k_{X/Y}(\log)[-i]
\]
for all $k$.
\end{Cor}

\begin{proof}
We shall prove an isomorphism of their sum
\[
\bigoplus_k Rf_*\Omega^k_{X/Y}(\log)[-k]
\cong \bigoplus_i \bigoplus_k R^if_*\Omega^k_{X/Y}(\log)[-i-k]
\]
By the $E_1$-degeneration of the spectral sequence associated to the 
Hodge filtration,
we have Lefschetz type isomorphisms for the direct sum
$\bigoplus_k Rf_*\Omega_{X/Y}^k(\log)[-k]$.
Therefore we obtain the assertion.
\end{proof}

We remark that Koll\'ar's result in \cite{Kollar} on the sheaf of 
top differential forms $\omega_X$ holds without the assumptions
on the well preparedness of the morphism, 
because $\omega_X$ behaves well under the birational morphisms
by the Grauert-Riemenschneider vanishing theorem.

By the above corollary, we have
\[
\bigoplus_k Rf_*\Omega^k_{B^h_I/Y}(\log)[-k]
\cong \bigoplus_i \bigoplus_k R^if_*\Omega^k_{B^h_I/Y}(\log)[-i-k].
\]
With respect to this decomposition, 
the boundary morphisms of the above complex are the sums of 
the Gysin homomorphisms
\[
R^qf_*\Omega^k_{B^h_I/Y}(\log) \to 
R^{q+1}f_*\Omega^{k+1}_{B^h_J/Y}(\log)
\]
for immersions $B^h_I \subset B^h_J$ of codimension $1$.
We note that underlying varieties of $B^h_{\emptyset}$ and $X$ are the same but
their boundaries are different.

There is a spectral sequence associated to the weight filtration
\begin{equation}\label{horizontal2}
\begin{split}
&E_1^{p,q} = \bigoplus_k H^{p+q-k}(\text{Gr}^W_{-p}(Rf_*\Omega^k_{X/Y}(\log))
\\
&= \bigoplus_k \bigoplus_{\text{codim }B^h_I = -p} 
R^{2p+q-k}f_*\Omega^k_{B^h_I/Y}(\log) 
\Rightarrow \bigoplus_k R^{p+q-k}f_*\Omega^k_{X/Y}(\log)
\end{split}
\end{equation}
where the differentials $d_1^{p,q}$ are nothing but the Gysin
homomorphisms above.
It degenerates at $E_2$ by \cite{Deligne-II}; 
we have $d_r^{p,q}=0$ for all $r \ge 2$.

\begin{proof}[Proof of Theorem~\ref{SP} continued]
We continue the proof of the semipositivity theorem in the case
where there are horizontal boundary components.
We use spectral sequences (\ref{horizontal1}) and (\ref{horizontal2}) 
associated to the weight filtrations, which will be denoted by ${}_IE_1^{p,q}$ 
and ${}_{II}E_1^{p,q}$.
The second one ${}_{II}E_1^{p,q}$ is obtained from the first one 
${}_IE_1^{p,q}$ by the decomposition with respect to the Hodge filtration:
\[
\bigoplus_k \text{Gr}^F_k({}_IE_1^{p,q} \vert_{Y \setminus C} 
\otimes \mathcal{O}_{Y \setminus C})
\cong {}_{II}E_1^{p,q} \vert_{Y \setminus C}.
\]
They degenerate at the $E_2$-terms.

Let $F$ denote either $f_*\Omega_{X/Y}^k(\log)$ or 
$R^kf_*\Omega_{X/Y}^n(\log)$. 
It has an induced weight filtration, and we have
\[
\text{Gr}^W(F) \cong 
\bigoplus_p F^{k+p}({}_{II}E_2^{p,k-p}) \quad \text{or} \quad 
\bigoplus_p F^{n+p}({}_{II}E_2^{p,k+n-p}). 
\]
Since $F^{k+p+1}({}_{II}E_2^{p,k-p}) = 0$ and 
$F^{n+p+1}({}_{II}E_2^{p,k+n-p}) = 0$, we deduce that
$\text{Gr}^W(F)$ is numerically semipositive, hence so is $F$.
\end{proof}


Graduate School of Mathematical Sciences, University of Tokyo,
Komaba, Meguro, Tokyo, 153-8914, Japan 

kawamata@ms.u-tokyo.ac.jp

\end{document}